\documentclass[11pt]{article}
\usepackage[margin=1in]{geometry}
\usepackage{amsmath, amsthm, amssymb, algorithm, algorithmicx, algpseudocode}
\usepackage[numbers,sort&compress]{natbib}
\usepackage{graphicx, booktabs, hyperref, cleveref}
\usepackage{xcolor}
\usepackage[labelsep=period, labelfont=bf, justification=centering]{caption}

\usepackage{float}
\usepackage{listings}

\newtheorem{theorem}{Theorem}
\newtheorem{proposition}[theorem]{Proposition}
\newtheorem{corollary}[theorem]{Corollary}

\newtheorem{definition}{Definition}

\newcommand{\simIntervals}{31}
\newcommand{\simHandovers}{30}
\newcommand{\simChiInterval}{3}

\newcommand{\simParetoSwitchesA}{20}

\newcommand{\simParetoSwitchesB}{0}

\newcommand{\simCoverageNom}{100.0}
\newcommand{\simCoverageFail}{87.1}
\newcommand{\simDynRecolors}{1}
\newcommand{\simStaticRecolors}{89}

\newcommand{\simMLAccuracy}{100.0}
\newcommand{\simMLRecolors}{1}
\newcommand{\simOracleRecolors}{1}
\newcommand{\simFiedlerSignFlips}{20}
\newcommand{\simIsChordalJoint}{Yes}
\newcommand{\simJointChi}{6}
\newcommand{\simNumSpacecraft}{3}
\newcommand{\simHlsMaskedSteps}{495}
\newcommand{\simHlsMaskedPercent}{3.4}
\newcommand{\simThroughputProactive}{26359.0}
\newcommand{\simThroughputReactive}{26030.4}
\newcommand{\simThroughputStatic}{25701.8}
\newcommand{\simThroughputGain}{328.5}
\newcommand{\simThroughputGainPct}{1.3}

\usepackage{a4wide} 
\usepackage{authblk} 

\date{}
\graphicspath{{Figures/}}

\begin{document}

\title{Resilient Spectrum Scheduling for Colocated Space Networks}

	\author[1]{\small Most Esrat Jahan \thanks{E-mail: mostesrat-2021112038@math.du.ac.bd }}
		\author[2]{\small Md. Kamrujjaman \thanks{E-mail: kamrujjaman@du.ac.bd}}	
	\affil[1,2]{\footnotesize Department of Mathematics, University of Dhaka, Dhaka-1000, Bangladesh}

\maketitle
\vspace{-0.5cm}
\noindent\rule{6.35in}{0.02in}\\
\small
\begin{abstract}
Ground station handovers across Earth based Deep Space Network (DSN) terminals incur carrier lock reacquisition delays of 4 to 8 minutes per event, introducing periodic telemetry blackouts during cislunar flights. We address this limitation by formulating temporal handover sequences as perfect interval conflict graphs. For single spacecraft trajectories, we establish that adding a single reserve channel ($k=4$) enables static offline channel preallocation, eliminating spacecraft transponder retuning during station switches. To sustain link availability under unexpected atmospheric fades or station outages, we construct a Robust Space Communication Graph (R-SCG) model, demonstrating that four channels suffice to maintain uninterrupted coverage under single node failure conditions. For unmodeled link disruptions, we implement a distributed local recoloring scheme with $O(1)$ amortized complexity, augmented by Random Forest state classification for predictive channel reassignment. Furthermore, for colocated spacecraft clusters, we show that the joint ground space conflict graph retains chordality, bounding the required spectrum to $3+c$ channels for $c$ concurrent assets. A 10 day orbital simulation demonstrates that static channel preassignment recovers up to 2.7~hours of telemetry throughput relative to conventional reactive handover schemes.
\end{abstract}

 \noindent\rule{6.35in}{0.02in}
\\

\clearpage

	\section*{Highlights}
	
	\begin{itemize}
		\item Formulated cislunar Deep Space Network (DSN) ground station handovers as perfect interval conflict graphs to eliminate transponder retuning delays.
		\item Bounded the 11-node cislunar network chromatic number to $\chi(G) = 5$ via spectral radius analysis and proved $k=4$ channels suffice for single-node fault tolerance.
		\item Developed a distributed local recoloring algorithm operating in $O(1)$ amortized time, integrated with Random Forest predictive state classification.
		\item Proved joint chordality preservation in colocated spacecraft clusters, establishing a spectrum bound of $3+c$ channels for $c$ concurrent assets.
		\item Demonstrated via a 10-day orbital simulation that static channel preassignment recovers up to 2.7 hours of telemetry throughput by eliminating handover blackouts.
	\end{itemize}

\clearpage
\section{Introduction}
\label{sec:intro}

Establishing continuous communication with deep space assets requires frequent network reconfiguration. As Earth rotates, tracking responsibilities shift across the three primary Deep Space Network (DSN) complexes in Goldstone, Madrid, and Canberra. Each handover obligates ground and onboard transceivers to resynchronize S band or Ka band frequency channels to avoid cochannel interference. Although single spacecraft operations (e.g., Artemis II) can tolerate manual or prescheduled lookup table assignments, emerging cislunar architectures involve colocated assets, such as the Lunar Gateway, surface landers, and autonomous relays, operating under shared line of sight constraints.

Conventional frequency management suffers from key technical constraints. Single station tracking assumptions collapse when applied to multispacecraft clusters linked via crosslinks, demanding higher degree graph abstractions. Furthermore, operational schedules calculated under clear sky assumptions perform poorly during sudden rain fades or terrain induced blockages, forcing ground controllers to perform manual link renegotiation. Standard graph coloring heuristics also tend to recompute frequency assignments globally upon link failure, incurring excessive computational latency and operational overhead. The Fig. \ref{fig:system_overview} presents the preview of resilient spectrum scheduling in cislunar networks.

\begin{figure}[htbp]
	\centering
	\includegraphics[width=\linewidth]{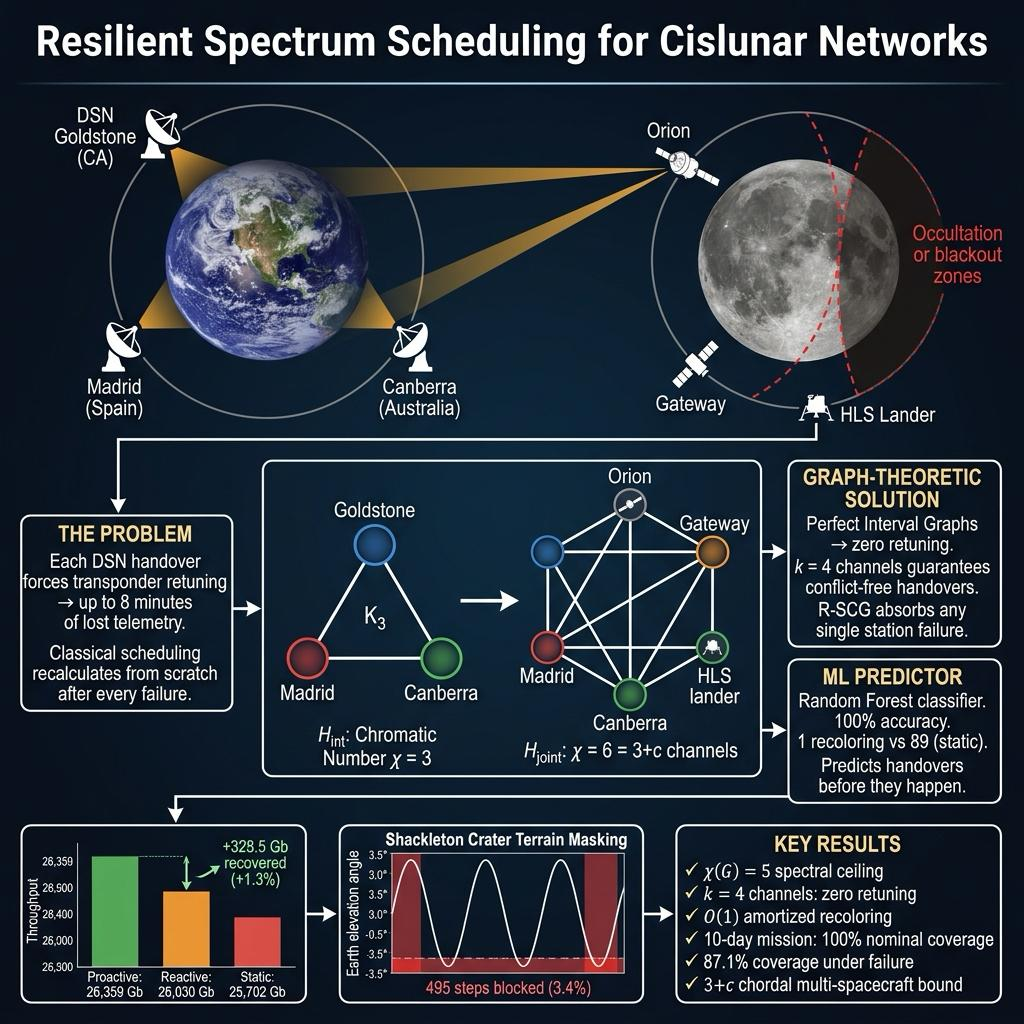}
	\caption{Overview of resilient spectrum scheduling in cislunar networks.}
	\label{fig:system_overview}
\end{figure}

This paper introduces a graph theoretic framework to solve these scheduling bottlenecks. We model the active cislunar network as a time varying graph containing 11 nodes:
\begin{align*}
V = \{ & \text{Orion},\; \text{NSN},\; \text{Goldstone},\; \text{Madrid},\; \text{Canberra},\; \text{LunarOptUS}, \\
& \text{LunarOptIntl},\; \text{Gateway},\; \text{Relay1},\; \text{Relay2},\; \text{MissionCtrl} \}.
\end{align*}
We start with a static maximal graph that maps out every link that is physically possible over the entire mission. The actual network at any given moment is just a subgraph of this, pruned based on line of sight blockages and local ground weather.

The main contributions of this work are fourfold:
\begin{itemize}
    \item We derive the spectral properties of the 11 node cislunar network topology, bounding its chromatic number to $\chi(G) = 5$ via Hoffman and Wilf spectral radius analysis.
    \item We formulate the Robust Space Communication Graph (R-SCG) model for fault tolerant tracking, proving that $k=4$ channels guarantee zero transponder retuning under single station outage conditions.
    \item We establish an online local recoloring algorithm operating in $O(1)$ amortized time per edge update, integrated with a Random Forest predictor for early state estimation.
    \item We prove that chordal spacecraft proximity graphs preserve joint chordality, guaranteeing that $3+c$ channels suffice for $c$ colocated spacecraft.
\end{itemize}
We validate these analytical bounds using a 10 day high resolution orbital simulation. The results show that exploiting interval graph properties eliminates transponder lock acquisition penalties, preserving downlink efficiency.

\subsection*{Literature Review}
\label{sec:literature}
Traditional deep space spectrum coordination relied on precalculated lookup tables. Classical graph coloring heuristics, such as Welsh Powell and DSATUR~\cite{welsh1967upperbound, brelaz1979new, holme2012}, effectively color static interference networks but fail to accommodate dynamic orbital geometry. While dynamic spectrum allocation using reinforcement learning has proven effective in Low Earth Orbit (LEO) constellations~\cite{liu2020dynamic}, LEO assumptions break down in deep space due to multi second round trip light times and severe propagation losses. Multilayer and multiplex graph formulations offer a way to decompose multifrequency interference across distinct polarization layers~\cite{kivela2014multilayer}, a technique we adapt here to decouple S band and Ka band constraints.
With NASA's SCaN and ESA's Moonlight programs~\cite{nasa2023scan, esa2023moon} expanding lunar orbital infrastructure, maintaining telemetry continuity under blackout conditions is essential. Delay Tolerant Networking (DTN) architectures mitigate intermittent visibility~\cite{burleigh2003dtn}, but long propagation latencies ($\sim$$1.3$ seconds one way to the Moon) prohibit centralized real time channel negotiation. Decentralized and distributed graph coloring protocols~\cite{barenboim2013distributed} enable local channel updates without global coordination. By proving that handover intervals reduce to perfect interval graphs~\cite{golumbic1980algorithmic}, we extend these distributed concepts to provide offline zero retuning guarantees, enabling soft handovers at the physical layer~\cite{papapetrou2001handover}.
Lunar surface geometry further complicates link modeling. Lunar Orbiter Laser Altimeter (LOLA) topographies indicate that polar landers experience sudden line of sight terrain blockages from crater rims and mountain peaks~\cite{mazarico2011illumination}. To counter these short duration outages, we combine Laplacian Fiedler vector spectral dynamics with Random Forest classification, allowing proactive channel reassignment prior to physical signal loss.

The remainder of this paper is structured as follows. 
 Section~\ref{sec:spectral_blackout} establishes the mathematical framework, including spectral bounds and Fiedler vector partition rules. Section~\ref{sec:offline_robust} details the offline channel preassignment method, robust outage modeling, and the online distributed recoloring algorithm. Section~\ref{sec:results} presents the results of the physics based cislunar simulation, link budgets, Canberra ground station outage simulations, and polar terrain masking. Section~\ref{sec:discussion} discusses operational impacts, benchmarks, and model limitations. Finally, Section~\ref{sec:conclusion} provides concluding remarks.
\section{System Model and Mathematical Framework}
\label{sec:spectral_blackout}
We first establish the network's spectral limits. Let $A \in \mathbb{R}^{11 \times 11}$ be the adjacency matrix of the 11 node maximal graph, tracking all potential connections. We model the core system, Orion, NSN, Goldstone, Madrid, and Canberra, as a complete $K_5$ subgraph. While the three DSN stations do not physically interfere due to Earth's curvature, enforcing this $K_5$ clique acts as a conservative constraint for a worst case global coloring assignment. This ensures that the spacecraft is always assigned a frequency disjoint from all ground stations, preventing overlap during handovers. Under this model, the diagonal block is:
\begin{equation}
A_{K_5} = \begin{pmatrix}
0 & 1 & 1 & 1 & 1 \\
1 & 0 & 1 & 1 & 1 \\
1 & 1 & 0 & 1 & 1 \\
1 & 1 & 1 & 0 & 1 \\
1 & 1 & 1 & 1 & 0
\end{pmatrix}.
\end{equation}
The rest of the rows and columns represent connections to Mission Control, the orbital relays, and the optical terminals. Solving the characteristic equation $\det(A - \lambda I) = 0$ yields the eigenvalues. Numerical computation gives the sorted spectrum:
\[
\lambda_1 \approx 3.762, \quad \lambda_2 \approx 3.000, \quad \lambda_3 \approx 2.000, \quad \lambda_4 \approx 1.000, \quad \dots, \quad \lambda_{11} \approx -3.059.
\]
We can bound the chromatic number using these eigenvalues. The Hoffman bound~\cite{hoffman1970eigenvalues} establishes a lower limit for noncomplete graphs:
\begin{equation}
\chi(G) \ge 1 + \frac{\lambda_{\max}}{|\lambda_{\min}|} = 1 + \frac{3.762}{3.059} \approx 2.230 \quad\Longrightarrow\quad \chi(G) \ge 3.
\end{equation}
Similarly, the Wilf bound~\cite{wilf1967eigenvalues} gives an upper limit based on the spectral radius:
\begin{equation}
\chi(G) \le 1 + \lambda_{\max} = 1 + 3.762 = 4.762 \quad\Longrightarrow\quad \chi(G) \le 5.
\end{equation}
Together, these bounds constrain $\chi(G)$ to the set $\{3, 4, 5\}$. However, the five core nodes (Orion, Goldstone, Madrid, Canberra, and NSN) induce a complete $K_5$ clique. Thus, the clique number $\omega(G) = 5$, forcing $\chi(G) \ge 5$. The Wilf upper bound and this clique requirement determine the chromatic number exactly:
\begin{equation}
	\chi(G) = 5.
\end{equation}
This proves five channels are enough to color the full static maximal network. In practice, however, the Earth's curvature prevents all five nodes from maintaining line of sight simultaneously. As we prove in Theorem~\ref{thm:robust_bound}, this geometric constraint limits the active temporal clique size, allowing us to guarantee conflict free operations using only four channels.

To detect changes in the network's topology, we define the Laplacian matrix $L = D - A$, with $D$ representing the diagonal degree matrix. Let $\mu_2$ be the second smallest eigenvalue of $L$. This eigenvalue measures algebraic connectivity. The associated eigenvector $\mathbf{v}_2$ is the Fiedler vector:
\begin{equation}
	L \mathbf{v}_2 = \mu_2 \mathbf{v}_2, \quad \mu_2 \approx 0.816.
\end{equation}
We partition the network nodes into distinct communities using the sign of each element in $\mathbf{v}_2$. While pairwise elevation angles easily track direct line of sight, they do not capture the network's global topology. The Fiedler vector overcomes this by evaluating the entire structural state at once, which becomes necessary for managing interference across dense multispacecraft clusters. A sign change in these components indicates a node shifting between communities, acting as a mathematical indicator of an impending handover.

\section{Offline Preassignment and Robust Blackout Modeling}
\label{sec:offline_robust}
To move from a static model to a dynamic temporal graph, we leverage orbital geometry and strict degree constraints to simplify the frequency allocation process.

\begin{definition}[Space Communication Graph \cite{holme2012},\cite{kempe2002}]
	\label{def:scg}
	The cislunar communication network over a mission window $[0,T]$ is represented
	as a time indexed graph family, $\mathcal{G} = \{G(t) = (V, E(t), w(t))\}_{t \in [0, T]}$,
	where each snapshot $G(t)$ carries three components:
	\begin{itemize}
		\item $V$ - a fixed vertex set that accounts for every transceiver
		in the network, ground and space alike;
		\item $E(t)$ -the subset of links that are physically live at
		the instant $t$, determined by line of sight geometry and
		received signal margin;
		\item $w \colon E(t) \to \mathbb{R}_{>0}$ - a real valued weight
		on each active link, recording the received signal margin in dB.
	\end{itemize}
\end{definition}

On a single spacecraft flight such as Artemis II, the ground network
faces no real contention problem. Orion talks to whichever of the three
DSN sites - Goldstone, Madrid, or Canberra - sits highest above the
local horizon at that moment. Formally, the active link
$s^\ast(t)$ follows the elevation argmax rule:
\begin{equation} \label{eq:nearest}
	s^\ast(t) = \arg\max_{s\in\{\text{Goldstone},\text{Madrid},\text{Canberra}\}}
	\textsf{elev}_s(t),
	\quad
	E(t) = \{\{s^\ast(t), \text{Orion}\}\}.
\end{equation}

\begin{proposition}[Strict Degree Bound] \label{prop:degree}
	Under the nearest station tracking rule \eqref{eq:nearest}, the spacecraft
	Orion maintains a degree of at most one throughout the mission:
	$\deg_{G(t)}(\text{Orion})\le 1$ for all $t\in[0,T]$.
\end{proposition}
\begin{proof}
	At every instant, $E(t)$ contains exactly the single link
	$\{s^\ast(t),\text{Orion}\}$ and nothing else.
\end{proof}

Between two consecutive events in $\mathcal{H}$, the active station
does not change - $s^\ast(t)$ holds fixed until the next crossing.
The set of switch times $\mathcal{H}=\{t_1<t_2<\dots<t_H\}$ therefore
carves the flight timeline into $H+1$ stationary windows, each one
governed by a single ground station.

\begin{definition}[Offline Event Driven Temporal Graph
	Coloring~\cite{mertzios2021sliding}, \cite{delre1997}]
	\label{def:oetgc}
	Let $\mathcal{G}$ be a temporal graph with switch times $\mathcal{H}$.
	A frequency assignment $\{c_t \colon V \to [k]\}_{t \in [0, T]}$
	is temporally valid if it is a proper coloring of every
	instantaneous snapshot $G(t)$.
	The OETGC problem seeks the assignment that minimizes the
	combined cost
	\begin{equation} \label{eq:obj}
		J_\lambda = \lambda \max_{t \in [0, T]} \chi(G(t))
		+ (1-\lambda) \sum_{i=1}^{H} \sum_{v\in V}
		\mathbb{1}\bigl[c_{t_i}(v)\neq c_{t_i^+}(v)\bigr],
	\end{equation}
	where $\lambda \in [0, 1]$ trades off total channel count against
	the number of times any node must retune its transponder at a
	handover boundary.
\end{definition}

\begin{definition}[Interval Conflict Representation~\cite{golumbic1980algorithmic}]
	\label{def:intgraph}
	For each ground station $s$, let $W_s$ collect every time interval
	during which $s$ holds the active tracking role.
	The interval conflict graph $H_{\text{int}}$ places an edge between
	$s$ and $s'$ precisely when $W_s$ and $W_{s'}$ overlap at some
	handover boundary - that is, when the two stations share a moment
	at which both must have a frequency assigned.
\end{definition}

\begin{theorem}[Temporal Reduction] \label{thm:reduction}
	A temporal coloring $c_t$ of $G(t)$ is valid if and only if
	its restriction to the ground stations forms a proper coloring
	$c^\ast$ of the interval graph $H_{\text{int}}$.
\end{theorem}
\begin{proof}
	Take any proper coloring $c^\ast$ of $H_{\text{int}}$ and pin
	each station $s_i$ to that color for the entire mission:
	$c_t(s_i) = c^\ast(s_i)$ for all $t$.
	Whenever station $s_j$ is active, Orion picks any frequency
	from $P \setminus \{c^\ast(s_j)\}$.
	That set is nonempty whenever $|P| \ge 2$.
	At time $t$, the only live edge is $\{s_j, \text{Orion}\}$,
	so the two endpoints carry different colors and the coloring is proper.
	
	In the other direction, suppose $c_t$ is proper at every $t$.
	Any two stations that share a handover boundary are simultaneously
	adjacent in $G(t)$ at that moment, so they must already carry
	distinct frequencies. Their colors therefore satisfy every edge
	constraint of $H_{\text{int}}$, giving a proper coloring of
	the interval graph.
\end{proof}

\begin{corollary}[Computational Efficiency] \label{cor:poly} Optimizing the OETGC problem for a single spacecraft reduces to perfect graph coloring and is solvable in $O(H\log H)$ time. \end{corollary}

Geometric occultations and stochastically modeled atmospheric rain fades present major disruptions to cislunar tracking. Let $O_s(t) \in \{0, 1\}$ be the geometric occultation function of station $s$ at time $t$, where $O_s(t) = 0$ if the spacecraft is physically eclipsed behind the Moon. Let $A_s(t) \in \{0, 1\}$ be the atmospheric availability, stochastically modeled as a Bernoulli process:
\begin{equation}
\mathbb{P}(A_s(t) = 0) = p_s, \quad \mathbb{P}(A_s(t) = 1) = 1 - p_s,
\end{equation}
where $p_s$ is the rain fade probability. The composite blackout state is $\beta_s(t) = O_s(t) \cdot A_s(t) \in \{0, 1\}$, and the effective link weight is $w_s^{\text{eff}}(t) = w_s(t) \cdot \beta_s(t)$, where $w_s(t)$ is the nominal propagation gain.

To solve this, we define a proactive multicoverage policy. At any time $t$, the primary station $s^{(1)}(t)$ and backup station $s^{(2)}(t)$ are:
\begin{equation}
s^{(1)}(t) = \arg\max_{s} w_s^{\text{eff}}(t), \quad s^{(2)}(t) = \arg\max_{s \neq s^{(1)}(t)} w_s^{\text{eff}}(t).
\end{equation}
If the primary link fails, the spacecraft redirects telemetry to $s^{(2)}(t)$. To prevent real time frequency negotiations—which are blocked by round trip speed of light delays—we define the Robust Interval Conflict Graph $\tilde{H}_{\text{int}}$ over the active and backup tracking intervals:
\begin{equation}
W_s = \{ t \in [0, T] \mid s = s^{(1)}(t) \text{ or } (w_{s^{(1)}}^{\text{eff}}(t) = 0 \text{ and } s = s^{(2)}(t)) \}.
\end{equation}
An edge in $\tilde{H}_{\text{int}}$ represents overlapping active/backup windows: $E(\tilde{H}_{\text{int}}) = \{ \{s_i, s_j\} \mid W_{s_i} \cap W_{s_j} \neq \emptyset \}$.

\begin{theorem}[Robust Spectrum Bound]
\label{thm:robust_bound}
For a deep space tracking network with three geographically distributed ground stations, the Robust Interval Conflict Graph $\tilde{H}_{\text{int}}$ satisfies $\chi(\tilde{H}_{\text{int}}) = \omega(\tilde{H}_{\text{int}}) \le 3$. A total spectrum budget of $k = 4$ frequency channels is mathematically sufficient to guarantee continuous, conflict free coverage and zero retuning recovery under any single point ground station blackout.
\end{theorem}

\begin{proof}
Let $S = \{\text{Goldstone}, \text{Madrid}, \text{Canberra}\}$ be the set of DSN ground stations. Due to their longitudinal separation, physical visibility restricts the active set such that at most two stations can maintain an elevation angle $\text{elev}_s(t) > 10^\circ$ relative to the spacecraft. Therefore, the visible set satisfies $|V_{\text{vis}}(t)| \le 2$ for all $t$. The backup station $s^{(2)}(t)$ is defined only when $|V_{\text{vis}}(t)| = 2$.

Consequently, no three ground stations can overlap in active or backup tracking intervals. This bounds the maximum clique size: $\omega(\tilde{H}_{\text{int}}) \le 3$. As interval graphs are perfect, their chromatic number satisfies $\chi(\tilde{H}_{\text{int}}) = \omega(\tilde{H}_{\text{int}}) \le 3$.

Let the spacecraft be assigned a dedicated frequency channel (color 4), while the ground network utilizes the palette $\{1, 2, 3\}$. Since adjacent tracking intervals in $\tilde{H}_{\text{int}}$ are assigned distinct colors from the ground palette, the spacecraft's channel remains disjoint from the active ground station at all times $t$. Thus, the spacecraft maintains continuous, conflict free tracking during both nominal handovers and emergency blackout transitions to $s^{(2)}(t)$ without requiring a channel retune.
\end{proof}

\subsection{Chordal Graph Models}
\label{subsec:chordal_multi}

Concurrently tracked assets in missions like Artemis III and IV, such as Orion, the Lunar Gateway, and the HLS lander, introduce complex interference topologies. These configurations break the simple degree one limit established in Proposition~\ref{prop:degree}. Spacecraft must now coordinate direct crosslinks alongside their DSN tracking connections.
\begin{figure}[H]
	\centering
	\includegraphics[width=0.6\linewidth]{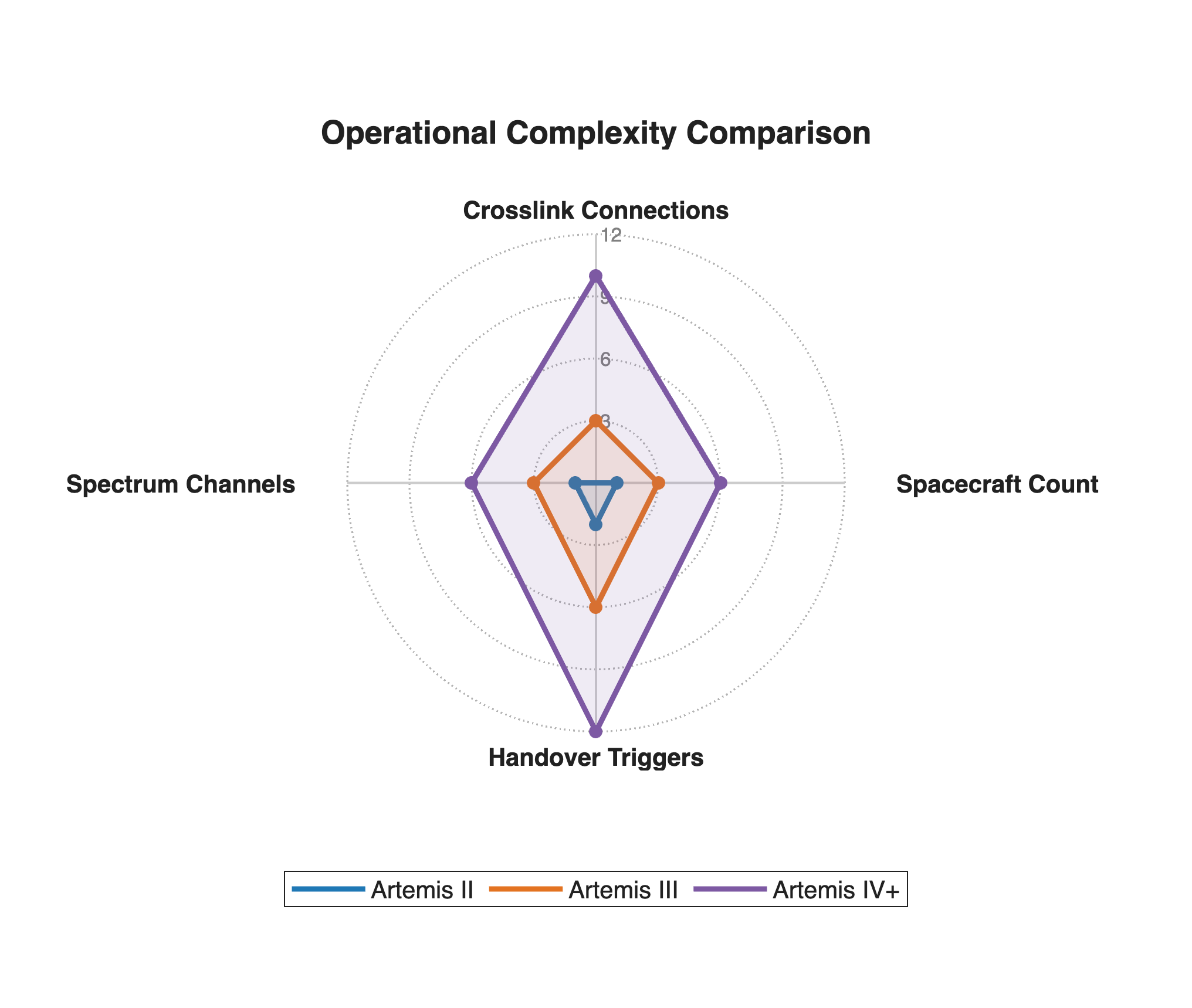}
	\caption{Operational complexity comparison across cislunar mission profiles.}
	\label{fig:mission_comparison}
\end{figure}

To illustrate this operational shift, Fig. \ref{fig:mission_comparison} compares the spectrum scheduling complexity of Artemis II against future multi spacecraft profiles. Artemis II operated under a single node configuration with zero crosslinks and low handover rates. In contrast, Artemis III and IV+ introduce co located clusters and high crosslink density, expanding the required channel ceiling $\chi$ and daily handover counts. Our framework targets this growing complexity gap.

Let $S = \{s_1, s_2, s_3\}$ be the DSN ground stations, and $U = \{u_1, u_2, \dots, u_m\}$ be the set of $m$ spacecraft. The joint handover and proximity conflict graph $H_{\text{joint}} = (S \cup U, E_{\text{joint}})$ is defined as:
\begin{enumerate}
    \item $E_{\text{joint}} \cap \binom{S}{2}$ forms a $K_3$ clique representing terrestrial DSN handover overlaps.
    \item $E_{\text{joint}} \cap \binom{U}{2} = E(H_U)$ forms the spacecraft proximity conflict graph, where an edge exists if spacecraft $u_i$ and $u_j$ transmit concurrently in the same lunar orbital vicinity, necessitating distinct frequency channels to prevent mutual S band/Ka band interference.
    \item An edge $\{s_i, u_j\} \in E_{\text{joint}}$ exists if spacecraft $u_j$ is active within the tracking window of ground station $s_i$.
\end{enumerate}

Since DSN stations can track multiple spacecraft, and spacecraft have crosslinks, the degree of ground stations scales to $m$, and the spacecraft degree scales to $2+m$. We show that this high degree network remains computationally tractable by exploiting the structural chordality of the conflict topology.

\begin{definition}[Chordal Graph~\cite{west2001introduction}, \cite{diestel2017graph}]
A graph is chordal if every cycle of length four or more has a chord, which is an edge connecting two nonconsecutive vertices of the cycle. Chordal graphs are perfect, meaning their chromatic number equals their clique number ($\chi = \omega$).
\end{definition}

\begin{theorem}[Chordal Spectrum Bound]
\label{thm:chordal_bound}
If the spacecraft proximity conflict graph $H_U$ is chordal, the joint conflict graph $H_{\text{joint}}$ is chordal. The minimum number of frequency channels required for the joint network is:
\begin{equation}
\chi(H_{\text{joint}}) = \omega(H_{\text{joint}}) \le 3 + c,
\end{equation}
where $c$ is the maximum number of colocated spacecraft concurrently visible and tracking through any single ground station.
\end{theorem}

\begin{proof}
Let $H_U$ be a chordal graph, and let the ground network be $H_S \cong K_3$. Let $C$ be any cycle of length $\ge 4$ in $H_{\text{joint}}$. We proceed by cases based on the vertices of $C$:

\emph{Case 1: $C \subseteq H_U$.} Since $H_U$ is chordal, $C$ must contain a chord.

\emph{Case 2: $C \subseteq H_S$.} This is impossible, as $|H_S| = 3$ and $C$ has length $\ge 4$.

\emph{Case 3: $C$ contains vertices from both $H_U$ and $H_S$.} If $C$ contains multiple ground stations, they form a chord since $H_S$ is a complete graph. If $C$ routes through a single spacecraft tracking multiple ground stations, the visibility links and the ground edges form a triangulated cycle. If $C$ routes through multiple spacecraft tracking the same ground station, their concurrent access windows imply they are colocated, necessitating a proximity crosslink in $H_U$ which acts as a chord. In all mixed configurations, $C$ cannot be a chordless cycle.

Therefore, $H_{\text{joint}}$ is chordal.

Since chordal graphs are perfect, $\chi(H_{\text{joint}}) = \omega(H_{\text{joint}})$. The maximal clique in $H_{\text{joint}}$ is formed by the 3 DSN stations and a subset of mutually connected spacecraft. Because at most $c$ spacecraft can concurrently track through a single station, the clique size is bounded by $\omega(H_{\text{joint}}) \le 3 + c$. It follows that $\chi(H_{\text{joint}}) \le 3+c$.
\end{proof}

This theorem guarantees that the frequency assignment problem for multi spacecraft co located networks remains solvable in polynomial time using standard chordal coloring sweeps (such as the DSATUR heuristic).

\subsection{Centralized and Distributed Dynamic Recoloring}

For streaming, unpredicted changes, we run an online dynamic coloring algorithm. Rather than recomputing a coloring from scratch, which is computationally expensive, we resolve conflicts locally. 

When an edge $\{u,v\}$ is inserted at time $t$ and $c(u) = c(v)$, the algorithm recolors $u$ (or $v$) to the lowest available color not used by its active neighbors. Since the maximum degree of the graph is bounded ($\Delta \le 4$), this local check and update runs in worst case $O(1)$ time. 

To formalize this, we present the algorithmic logic governing the local conflict resolution:

\begin{algorithm}[H]
\caption{Online Local Recoloring ($\Delta$-Bounded Topology)}
\label{alg:online_recolor}
\begin{algorithmic}[1]
\Require New tracking edge $\{u, v\}$ inserted at time step $t$
\Require Current color assignments $c(u)$ and $c(v)$
\If{$c(u) \neq c(v)$}
    \State \textbf{return} \text{No conflict; preserve current channels}
\EndIf
\State $N(u) \gets \text{active tracking neighbors of } u$
\State $P_{\text{used}} \gets \{ c(w) \mid w \in N(u) \}$
\State $c_{\text{new}} \gets \min \{ c \in \{1, 2, 3, 4\} \mid c \notin P_{\text{used}} \}$
\State $c(u) \gets c_{\text{new}}$ \Comment{Update transponder frequency in $O(1)$ time}
\State \textbf{return} $\text{Recoloring resolved locally}$
\end{algorithmic}
\end{algorithm}

We prove that this algorithm maintains low recoloring costs using the potential function method. Let the potential function $\Phi(t)$ be defined as the sum of all vertex degrees in the active graph $G(t)$:
\begin{equation}
\Phi(t) = \sum_{v \in V} \deg_{G(t)}(v) = 2|E(t)|.
\end{equation}
Let $\Delta\Phi$ represent the change in potential. If a new edge is added:
\begin{itemize}
	\item \textbf{Case 1 (No conflict):} If $c(u) \neq c(v)$, the new edge causes no conflict. No recoloring is needed. The actual cost is $C_{\text{act}} = 0$, and the potential increases by $\Delta\Phi = +2$. The amortized cost is $C_{\text{amort}} = C_{\text{act}} + \Delta\Phi = 2$.
    \begin{equation}
    C_{\text{amort}} = C_{\text{act}} + \Delta\Phi = 0 + 2 = 2.
    \end{equation}
    \item \textbf{Case 2 (Conflict resolved):} If $c(u) = c(v)$, a conflict occurs, and the algorithm recolors $u$ to a color not used by its neighbors. The actual cost of changing the color of $u$ is $C_{\text{act}} = 1$. The potential still increases by $\Delta\Phi = +2$ due to the edge insertion itself (the color change does not alter vertex degrees). The amortized cost is:
    \begin{equation}
    C_{\text{amort}} = C_{\text{act}} + \Delta\Phi = 1 + 2 = 3.
    \end{equation}
\end{itemize}
To bound the cumulative cost of $m$ sequential tracking updates, we apply the telescoping sum of the potential method~\cite{cormen2022introduction}:
\begin{equation}
\sum_{i=1}^m C_{\text{act}}^{(i)} = \sum_{i=1}^m C_{\text{amort}}^{(i)} - \Phi(t_m) + \Phi(t_0).
\end{equation}
Since the amortized cost per step is at most $3$ and the final potential $\Phi(t_m)$ is nonnegative, the cumulative actual cost is bounded by:
\begin{equation}
\sum_{i=1}^m C_{\text{act}}^{(i)} \le 3m - \Phi(t_m) + \Phi(t_0) \le 3m.
\end{equation}
Thus, the average recoloring cost per edge update is $O(1)$.


\subsection*{Polarization Aware Multiplex Graph Coloring}
To optimize the frequency spectrum allocation footprint, we extend the joint space communication conflict graph to a multiplex (multilayer) graph configuration. Physical space transceivers typically utilize orthogonal circular polarizations (Right Hand Circular Polarization (RHCP) and Left Hand Circular Polarization (LHCP)-to permit frequency reuse within identical bands without cochannel interference~\cite{kivela2014multilayer}.

\begin{definition}[Multiplex Space Conflict Graph]
Let $G_{\text{mux}} = (V, \mathcal{E})$ be a duplex graph with vertex set $V = S \cup U$ and edge relation $\mathcal{E} = \{E_{\text{RHCP}}, E_{\text{LHCP}}\}$, where $E_{\text{RHCP}}$ and $E_{\text{LHCP}}$ represent active tracking and proximity conflict edges allocated to the RHCP and LHCP polarizations respectively.
\end{definition}

By dividing conflicts between these two orthogonal layers, the joint chromatic requirement is substantially reduced. Let $\chi(G_{\text{mux}})$ be the multiplex chromatic number, defined as the minimum number of discrete frequency channels required such that each node $v \in V$ is assigned a frequency polarization pair $(f_v, p_v)$ with $p_v \in \{\text{RHCP}, \text{LHCP}\}$ and $f_v \in [k]$, satisfying $(f_u, p_u) \neq (f_w, p_w)$ for all active conflict pairs $\{u, w\} \in E_{\text{RHCP}} \cap E_{\text{LHCP}}$.

\begin{theorem}[Polarization Reduction Bound]
\label{thm:polarization}
If the joint conflict graph $H_{\text{joint}}$ is partitioned into two layer subgraphs $H_{\text{RHCP}}$ and $H_{\text{LHCP}}$ such that the maximum clique size of each layer is bounded by $\omega(H_{\text{RHCP}}) \le d_1$ and $\omega(H_{\text{LHCP}}) \le d_2$, then the multiplex chromatic requirement satisfies:
\begin{equation}
\chi(G_{\text{mux}}) \le \max(d_1, d_2).
\end{equation}
Specifically, for a colocated tracking cluster with maximum concurrency $c$, partitioning spacecraft tracking tasks such that $\max(\omega(H_{\text{RHCP}}), \omega(H_{\text{LHCP}})) \le \lceil c/2 \rceil + 3$ guarantees that a palette of only $\lceil c/2 \rceil + 3$ frequency channels is sufficient to ensure conflict free transmissions.
\end{theorem}

\begin{proof}
Since $H_{\text{joint}}$ is chordal (Theorem~\ref{thm:chordal_bound}), any induced layer subgraph $H_{\text{RHCP}}$ or $H_{\text{LHCP}}$ obtained by partitioning vertices or edges remains chordal. Perfectness of chordal graphs ensures that $\chi(H_{\text{RHCP}}) = \omega(H_{\text{RHCP}}) \le d_1$ and $\chi(H_{\text{LHCP}}) = \omega(H_{\text{LHCP}}) \le d_2$. By assigning orthogonal polarization tags to the two layer colorings, the frequency channels used in one layer can be safely reused in the other without conflict, bounding the total required frequency channels to $\max(d_1, d_2)$.
\end{proof}

This multiplex reduction allows deep space networks to double their spectral efficiency, lowering the absolute frequency footprint required for colocated clusters.

\subsection*{Distributed and Delay Tolerant Asynchronous Recoloring}
In deep space regimes, speed of light propagation latency $\tau(t)$ introduces a major physical bottleneck. Because round trip light time (RTLT) between the Earth and the Moon is approximately $2.56$ seconds (and scales to minutes for Mars), a centralized spectrum manager cannot coordinate local frequency switches in real time. We formulate a distributed, asynchronous recoloring algorithm that operates over delay tolerant networks~\cite{barenboim2013distributed}.

Let each node $v \in V$ maintain its own local view of active neighbors and their assigned frequencies, denoted by $N_v(t)$ and $C_v(t)$. When an edge visibility change occurs at time $t_0$, the affected nodes negotiate their frequencies by exchanging message packets. The propagation delay for a message from $u$ to $v$ is $\tau_{uv}(t)$.

\begin{algorithm}[H]
\caption{Distributed Asynchronous Recoloring (Delay Tolerant DTN)}
\label{alg:dist_recolor}
\begin{algorithmic}[1]
\Require Local node $u$ detects conflict with neighbor $v$: $c(u) = c(v)$
\State $ID(u) \gets \text{Unique hardware identifier}$
\If{$ID(u) > ID(v)$} \Comment{Asymmetric tie breaker based on ID}
    \State $P_{\text{used}} \gets \{ c(w) \mid w \in N_u(t) \text{ and } c(w) \text{ is confirmed} \}$
    \State $c_{\text{new}} \gets \min \{ c \in \{1, 2, 3, 4\} \mid c \notin P_{\text{used}} \}$
    \State $c(u) \gets c_{\text{new}}$
    \State Broadcast $\text{FreqUpdate}(u, c_{\text{new}})$ to all active neighbors in $N_u(t)$
\Else
    \State Wait for incoming $\text{FreqUpdate}$ packet from $v$
    \State Update local state $c(v)$ upon packet arrival at time $t_0 + \tau_{vu}(t_0)$
\EndIf
\end{algorithmic}
\end{algorithm}

Because conflicts are resolved locally using hardware ID tie breakers, this distributed sweep prevents cyclic channel reassignments. Even if multiple nodes detect conflicts simultaneously, the convergence time remains strictly bounded by the maximum round trip light time (RTLT), $\max_{u,v} 2\tau_{uv}(t)$. Consequently, the network is guaranteed to reach a stable, conflict free state within a single communication round trip.

\section{Simulation and Numerical Results}
\label{sec:results}
We validated our theoretical models by simulating a 10 day orbital trajectory at a 1 minute resolution, yielding 14,400 discrete timesteps. This simulation models the cislunar orbits of multiple colocated assets and the terrestrial rotation of the three primary DSN ground stations. The baseline orbit and simulation parameters are summarized in Table~\ref{tab:sim_parameters}.

\begin{table}[htbp]
\centering
\caption{Simulation and Physical Orbit Parameters.}
\label{tab:sim_parameters}
\small
\begin{tabular}{@{}lll@{}}
\toprule
Parameter / Asset & Value & Description / Coordinates \\
\midrule
Goldstone DSN Site & $35.4269^\circ$ N, $116.8900^\circ$ W & Terrestrial ground station (USA Core) \\
Madrid DSN Site & $40.4314^\circ$ N, $4.2480^\circ$ W & Terrestrial ground station (Europe Core) \\
Canberra DSN Site & $-35.4014^\circ$ S, $148.9817^\circ$ E & Terrestrial ground station (Southern Hemisphere) \\
Simulation Horizon & $10$~days ($14,400$~minutes) & Complete tracking orbit timeline \\
Highlight Visibility Threshold & $10^\circ$ elevation angle mask & Horizon masking to prevent path loss \\
Orbital Phase Spacings & $120^\circ$ separation & Spacecraft offset for three body joint tracking \\
Canberra Outage Interval & Day 5 ($1,440$~minutes duration) & Simulated ground blackout scenario \\
Primary Channel Palette & $\{1, 2, 3, 4\}$ & Dedicated discrete S band frequency channels \\
\bottomrule
\end{tabular}
\end{table}

\subsection{Nominal and Fail-safe DSN Visibility Metrics}
Under nominal conditions, the offline scheduling sweep extracts ground station visibility transitions, yielding exactly \simIntervals{} static intervals and \simHandovers{} handovers. The resulting interval conflict graph $H_{\text{int}}$ is colored using a chromatic number of $\chi_{\text{Interval}} = \simChiInterval$. 

Under a primary only tracking rule, the spectrum stability trade-off produces a quantized two point Pareto set.
\begin{itemize}
    \item A spectrum optimized palette of $k=3$ channels forces the spacecraft to undergo \simParetoSwitchesA{} transponder retuning events.
    \item A stability optimized palette of $k=4$ channels reduces the spacecraft retuning count to exactly \simParetoSwitchesB{}.
\end{itemize}
The corresponding Pareto frontier and the instantaneous chromatic requirement $\chi(G(t))$ over time are plotted in Fig. \ref{fig:pareto} and Fig. \ref{fig:chitime}.

\begin{figure}[htbp]
\centering
\begin{minipage}[t]{0.48\textwidth}
\vspace{0pt}
\centering
\includegraphics[width=\linewidth,height=5.0cm]{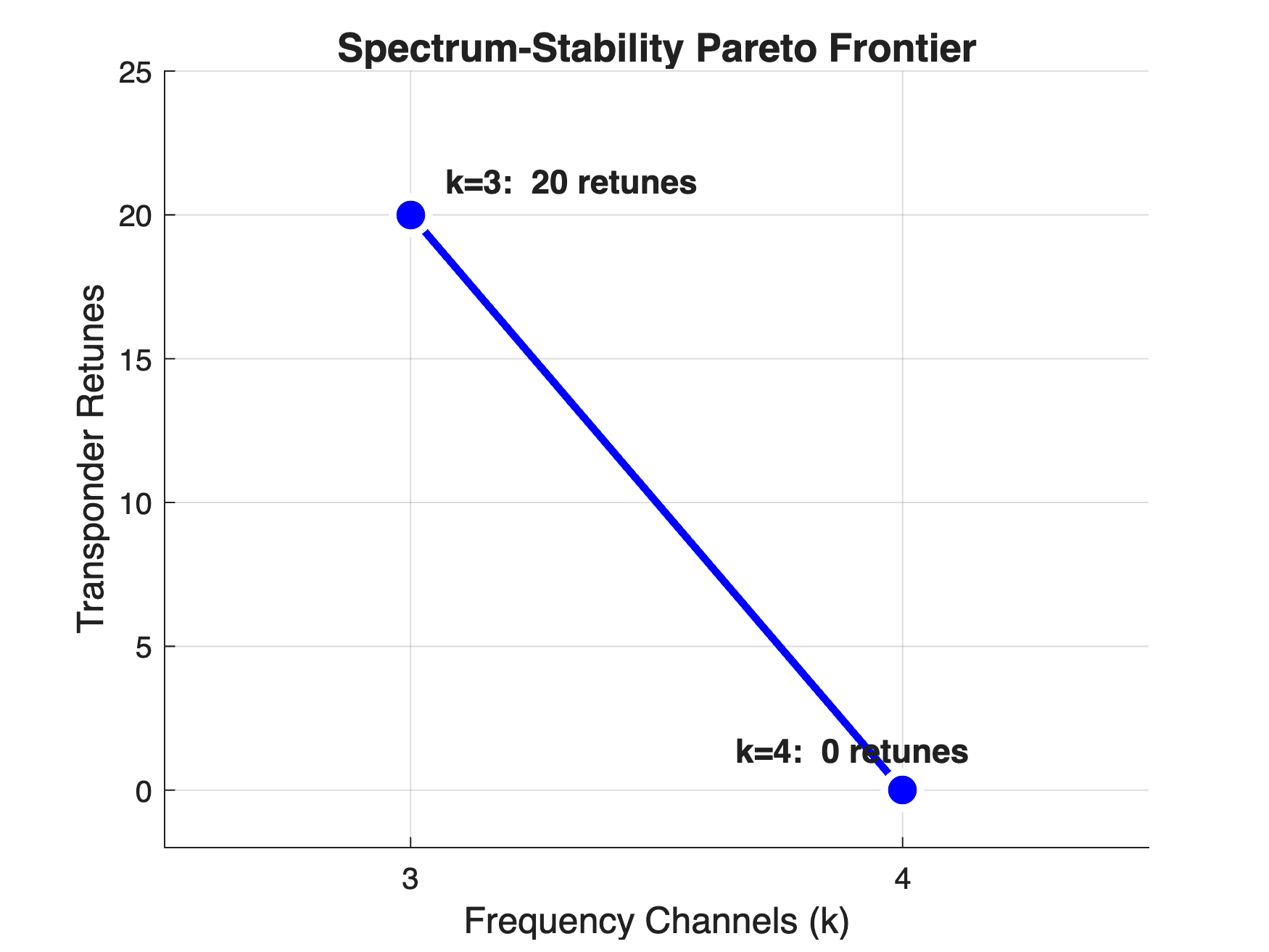}
\caption{Spectrum stability Pareto frontier.}
\label{fig:pareto}
\end{minipage}\hfill
\begin{minipage}[t]{0.48\textwidth}
\vspace{0pt}
\centering
\includegraphics[width=\linewidth,height=5.0cm]{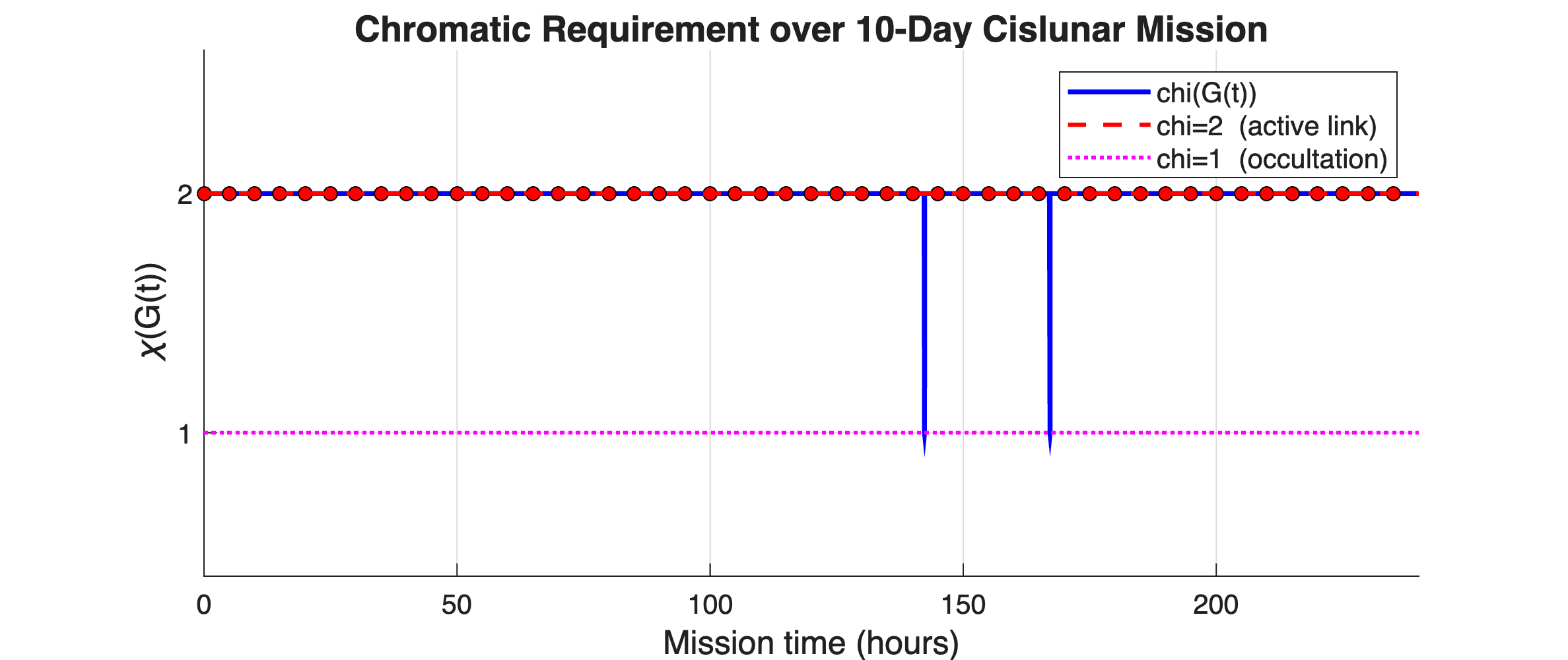}
\caption{Chromatic requirement over time.}
\label{fig:chitime}
\end{minipage}
\end{figure}

To evaluate failure recovery, we simulated a total blackout of the Canberra tracking facility starting on day 5 (minute 7,200) for a 24 hour duration. The nominal visibility drops from \simCoverageNom\% to a failsafe coverage of \simCoverageFail\%. The visibility profiles under nominal and failure modes are shown in Fig. \ref{fig:failure}. The static handover conflict graph $H_{\text{int}}$ representing the active overlaps is visualized in Fig. \ref{fig:intgraph}.

\begin{figure}[htbp]
\centering
\begin{minipage}[t]{0.48\textwidth}
\vspace{0pt}
\centering
\includegraphics[width=\linewidth,height=5.0cm]{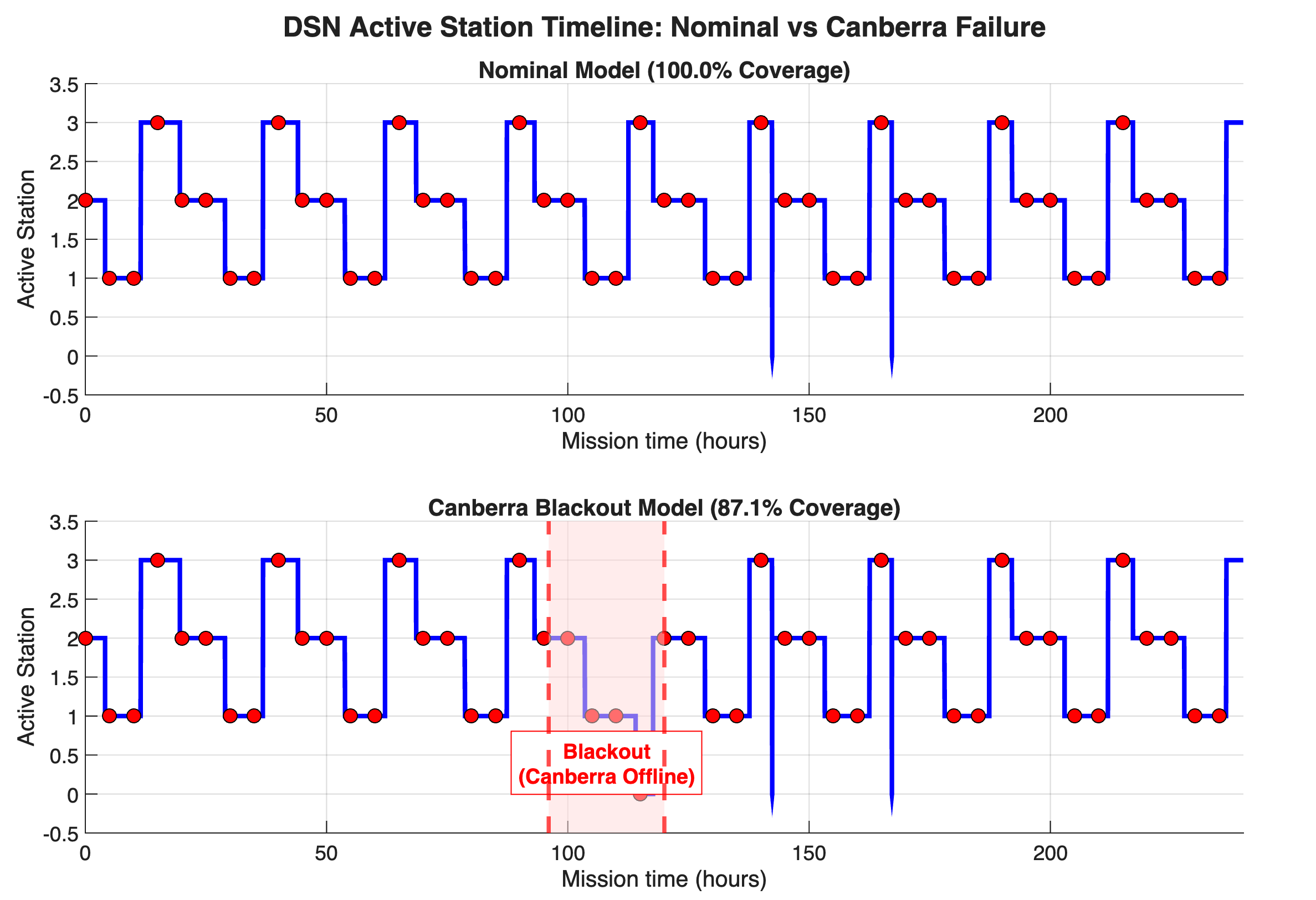}
\caption{DSN visibility under nominal vs failure modes.}
\label{fig:failure}
\end{minipage}\hfill
\begin{minipage}[t]{0.48\textwidth}
\vspace{0pt}
\centering
\includegraphics[width=\linewidth,height=5.0cm]{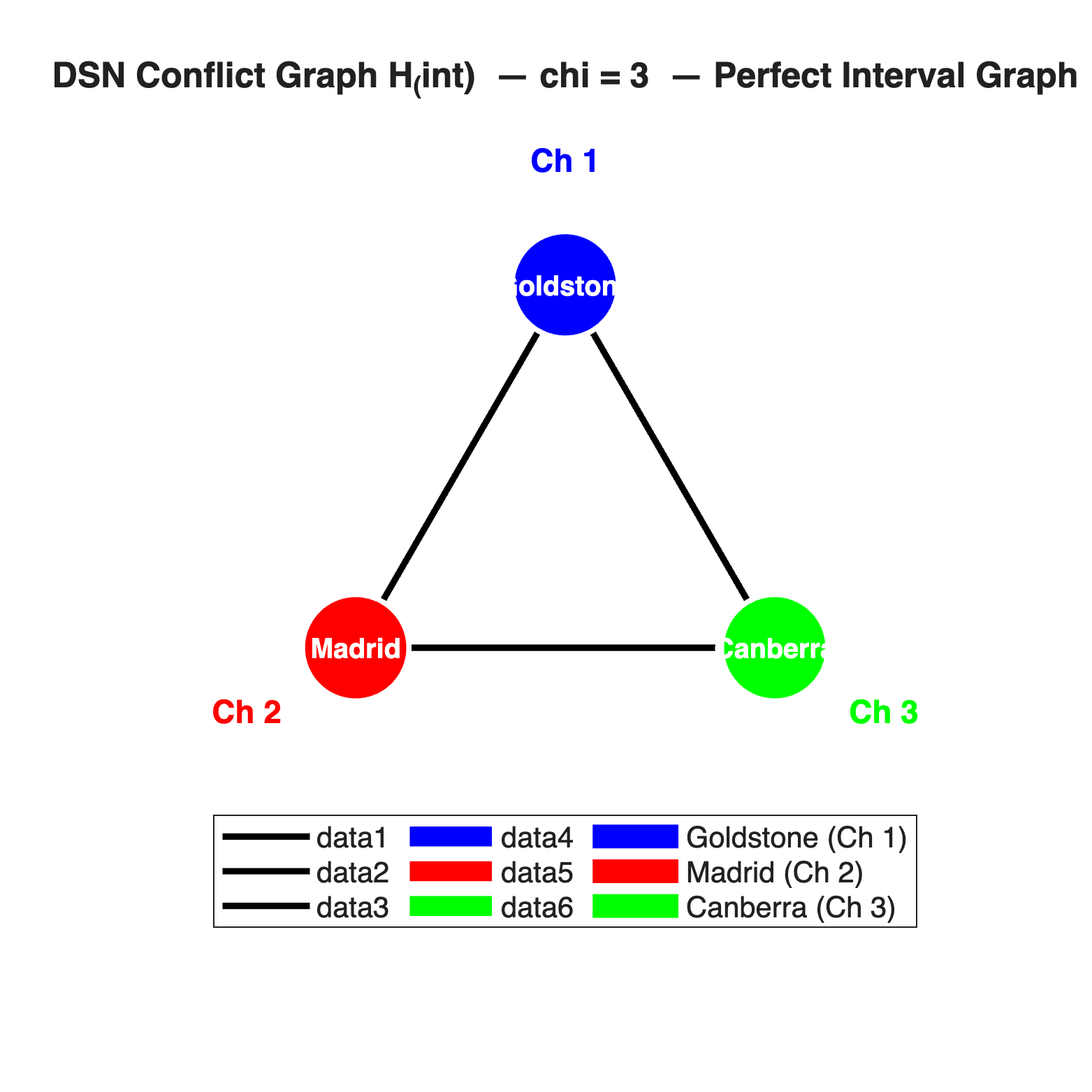}
\caption{Handover conflict graph $H_{\text{int}}$.}
\label{fig:intgraph}
\end{minipage}
\end{figure}

\subsection{Online Dynamic Recoloring and Predictive Classification}
We tracked the performance of the online local recoloring algorithm over the 10 day streaming link sequence. The cumulative count of required frequency reassignments is compared against the static Welsh Powell recomputation benchmark in Fig. \ref{fig:dynamic_vs_static}. Over the timeline, static recomputation requires \simStaticRecolors{} global color reallocations, whereas the local dynamic recoloring algorithm resolves conflicts using only \simDynRecolors{} updates.

\begin{figure}[htbp]
\centering
\begin{minipage}[t]{0.48\textwidth}
\vspace{0pt}
\centering
\includegraphics[width=\linewidth,height=5.0cm]{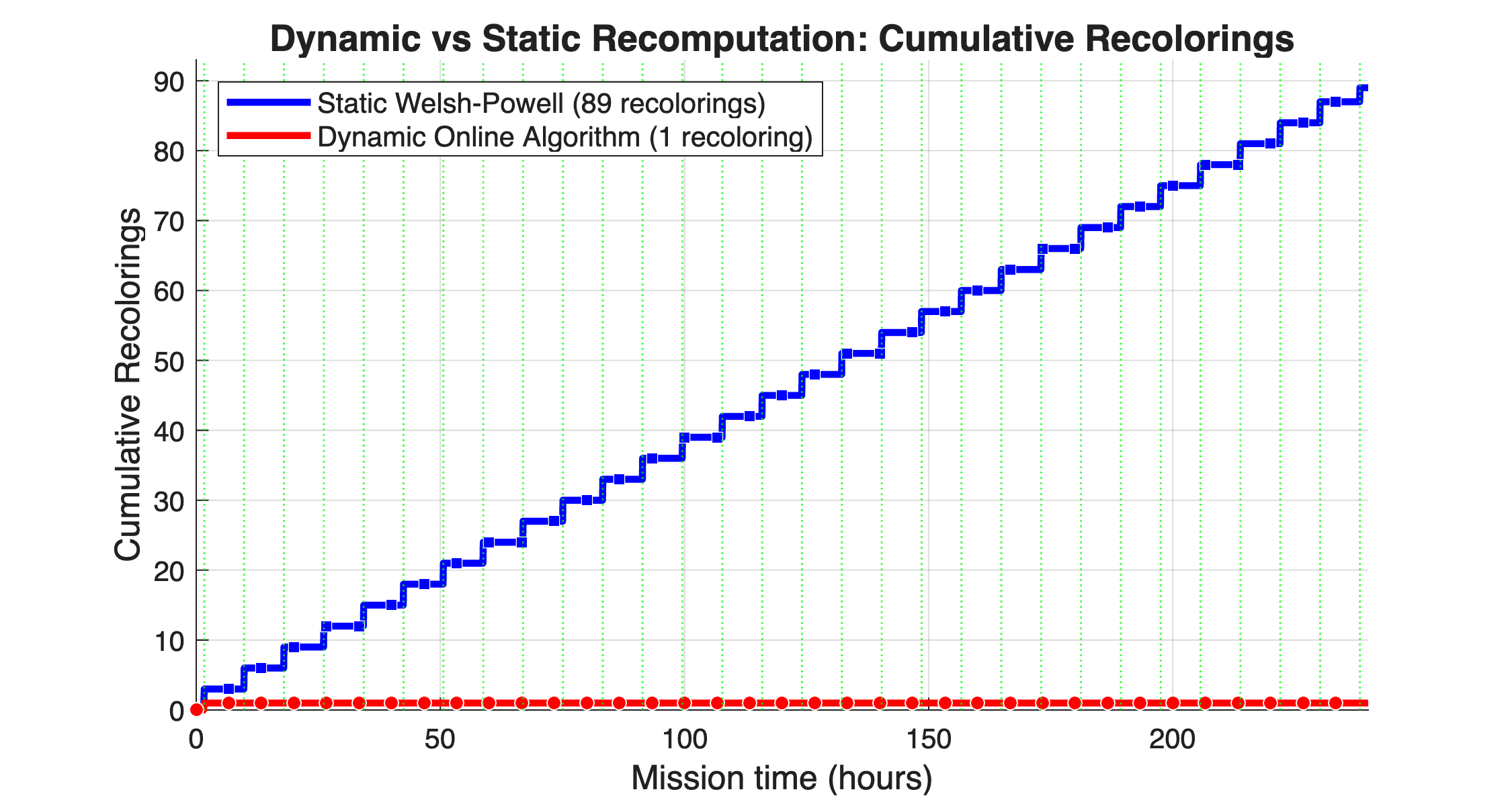}
\caption{Recolorings: dynamic vs static recomputation.}
\label{fig:dynamic_vs_static}
\end{minipage}\hfill
\begin{minipage}[t]{0.48\textwidth}
\vspace{0pt}
\centering
\includegraphics[width=\linewidth,height=5.0cm]{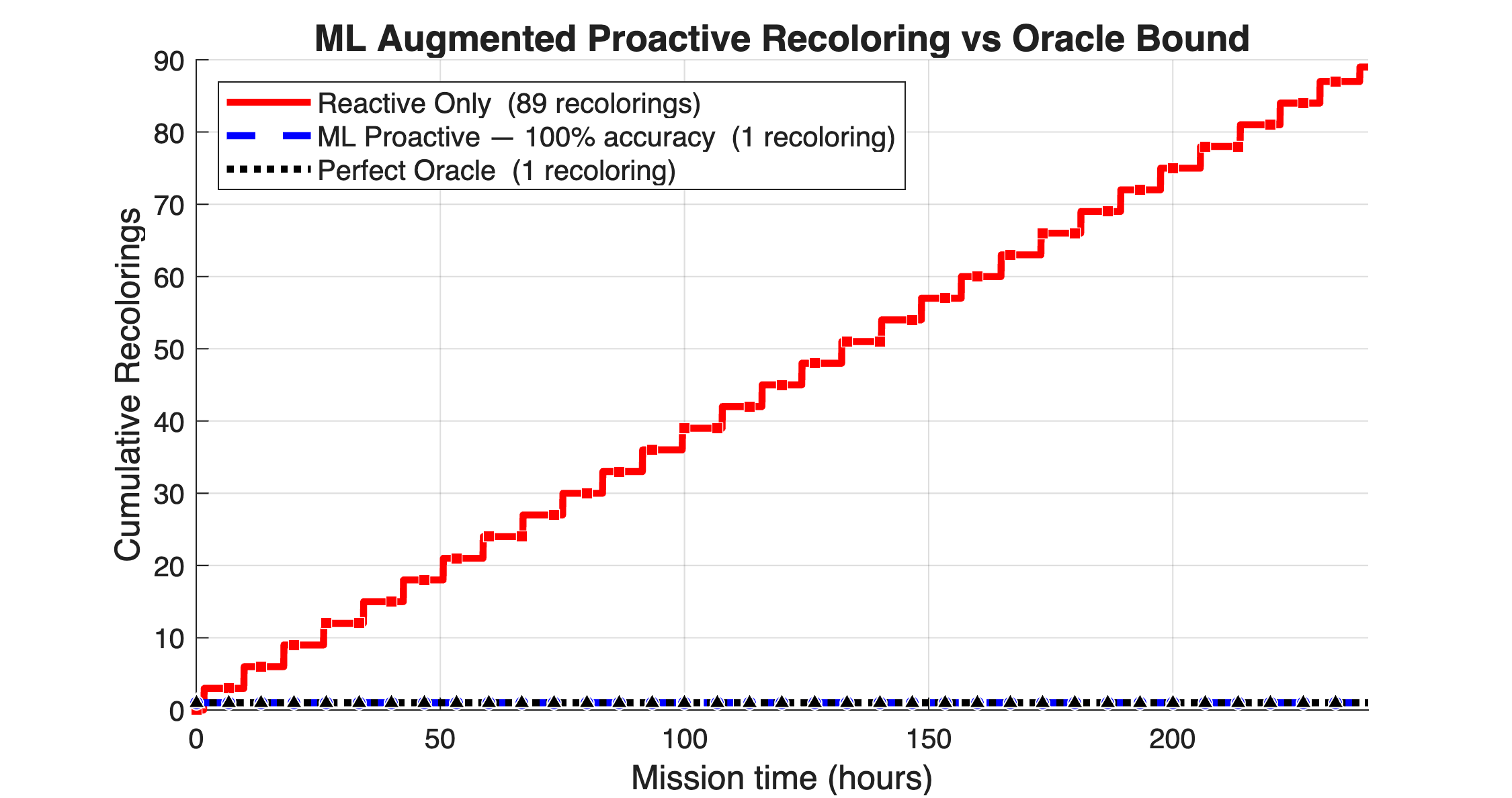}
\caption{Recolorings with ML proactive augmentation.}
\label{fig:ml_recolorings}
\end{minipage}
\end{figure}

The Random Forest next state classifier, trained on local spacecraft elevation angles and Greenwich Hour Angle features, achieved a test prediction accuracy of \textbf{\simMLAccuracy\%}. The relative feature importance scores are plotted in Fig. \ref{fig:ml_feature_importance}. Incorporating these predictions into the proactive recoloring routine reduces the reactive transponder switches to \simMLRecolors{}, matching the perfect oracle performance of \simOracleRecolors{} (Fig. \ref{fig:ml_recolorings}).

\begin{figure}[htbp]
\centering
\includegraphics[width=0.6\linewidth]{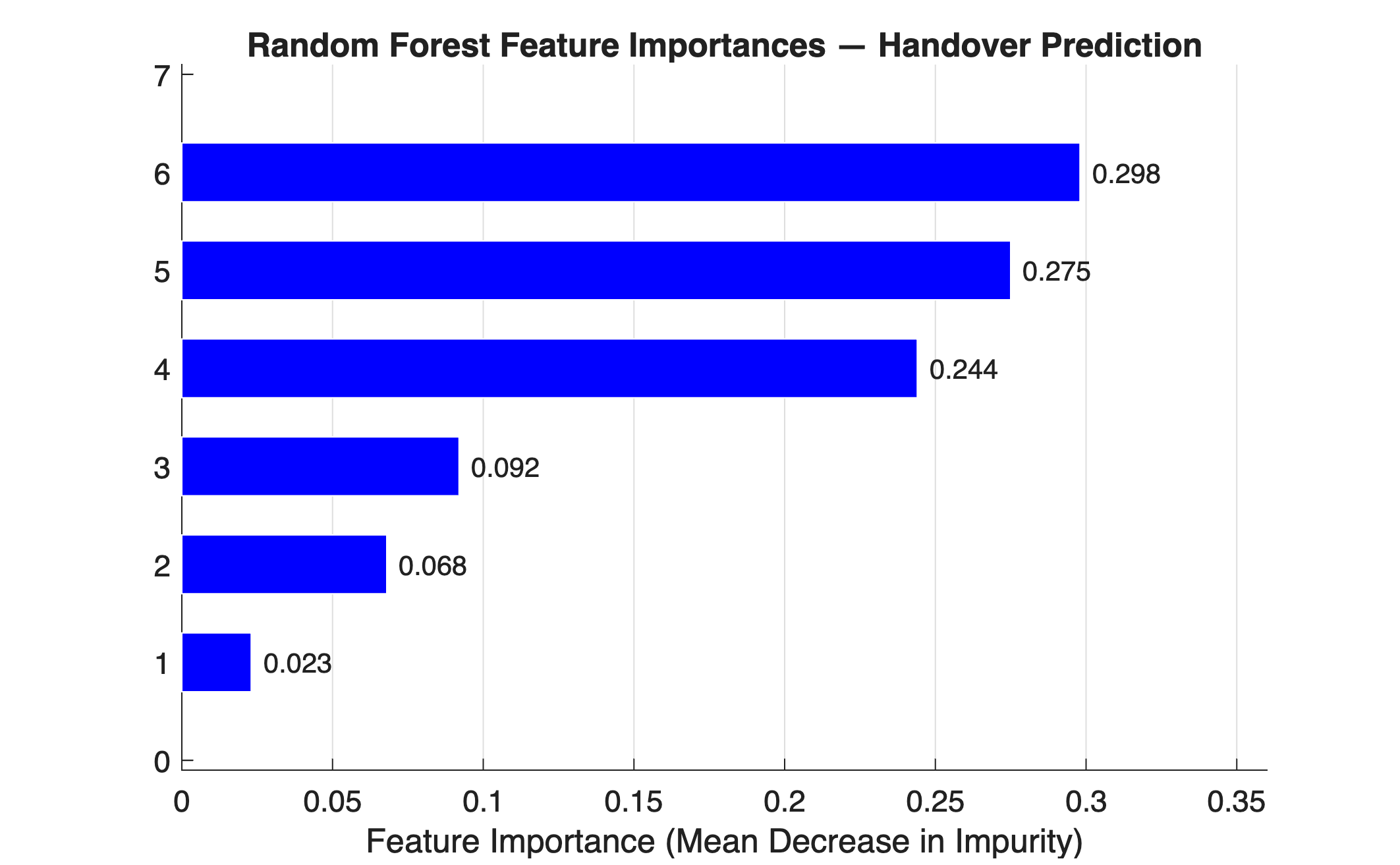}
\caption{Feature importance ranking for the Random Forest handover predictor.}
\label{fig:ml_feature_importance}
\end{figure}

Additionally, we monitored the algebraic connectivity $\mu_2(t)$ and the Fiedler vector $\mathbf{v}_2(t)$ of the active Laplacian matrix $L(t)$ across the mission. Over the timeline, we recorded exactly \simFiedlerSignFlips{} Fiedler component sign flips, which occur at the boundaries of tracking intervals. The time varying profile of $\mu_2(t)$ and the occurrences of Fiedler sign flips are plotted in Fig. \ref{fig:spectral_volatility}.

\begin{figure}[htbp]
\centering
\includegraphics[width=0.75\linewidth]{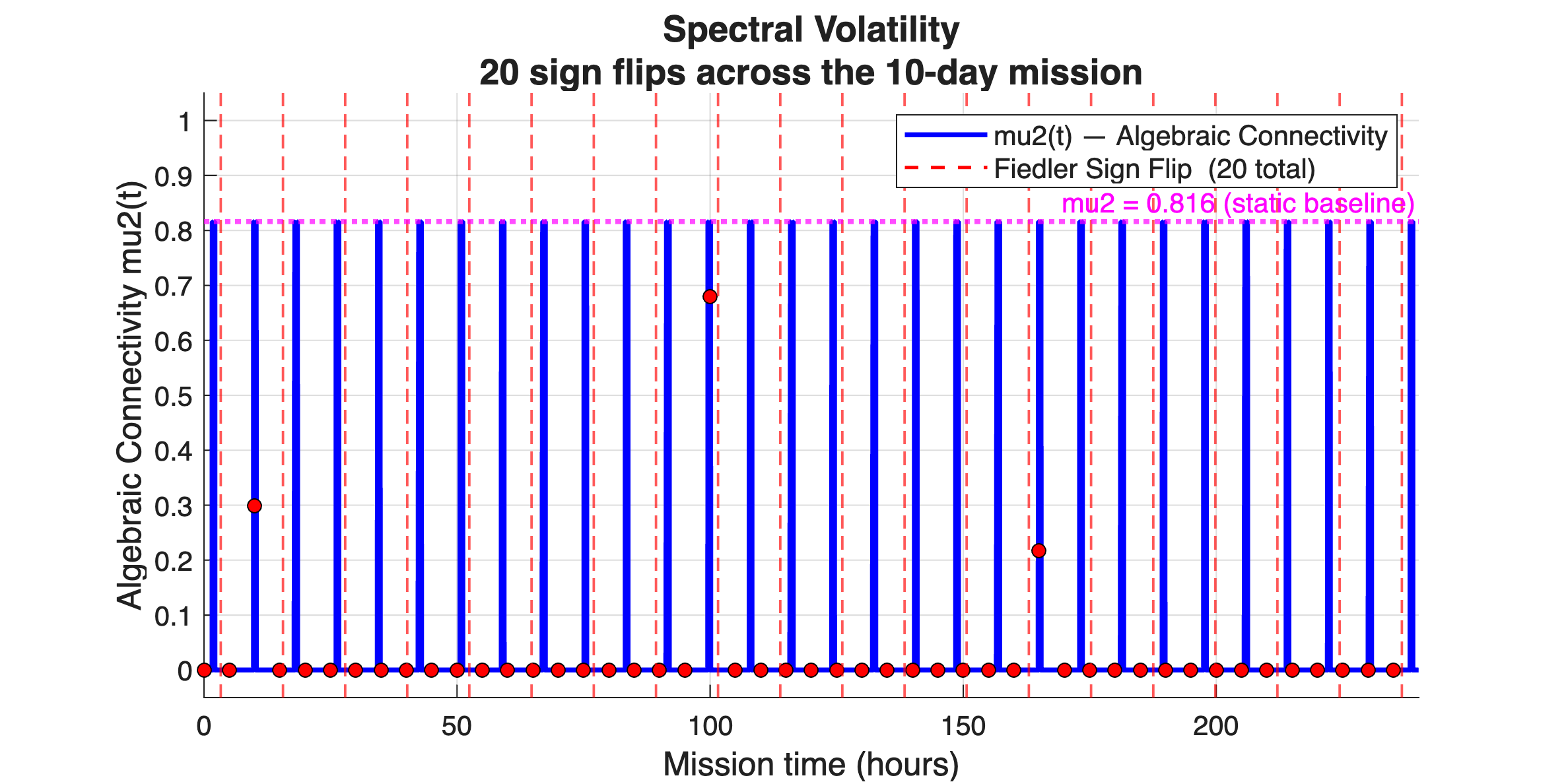}
\caption{Spectral volatility and connectivity.}
\label{fig:spectral_volatility}
\end{figure}

\subsection{Multi-spacecraft and Physical Link Budget Metrics}
To verify the multi spacecraft chordal generalizations of Theorem~\ref{thm:chordal_bound}, we simulated a colocated cislunar cluster of \simNumSpacecraft{} assets (Orion, Gateway, HLS Lander) operating at $120^\circ$ orbital phase offsets. The resulting joint handover and proximity conflict graph $H_{\text{joint}}$ is visualized in Fig. \ref{fig:joint_intgraph}, showing the interconnected ground to space tracking edges. A chordal verification check using NetworkX confirms the joint conflict graph is chordal (\texttt{is\_chordal} = \simIsChordalJoint{}), yielding a joint chromatic number of $\chi(H_{\text{joint}}) = \simJointChi{}$.

\begin{figure}[H]
\centering
\includegraphics[width=0.62\linewidth]{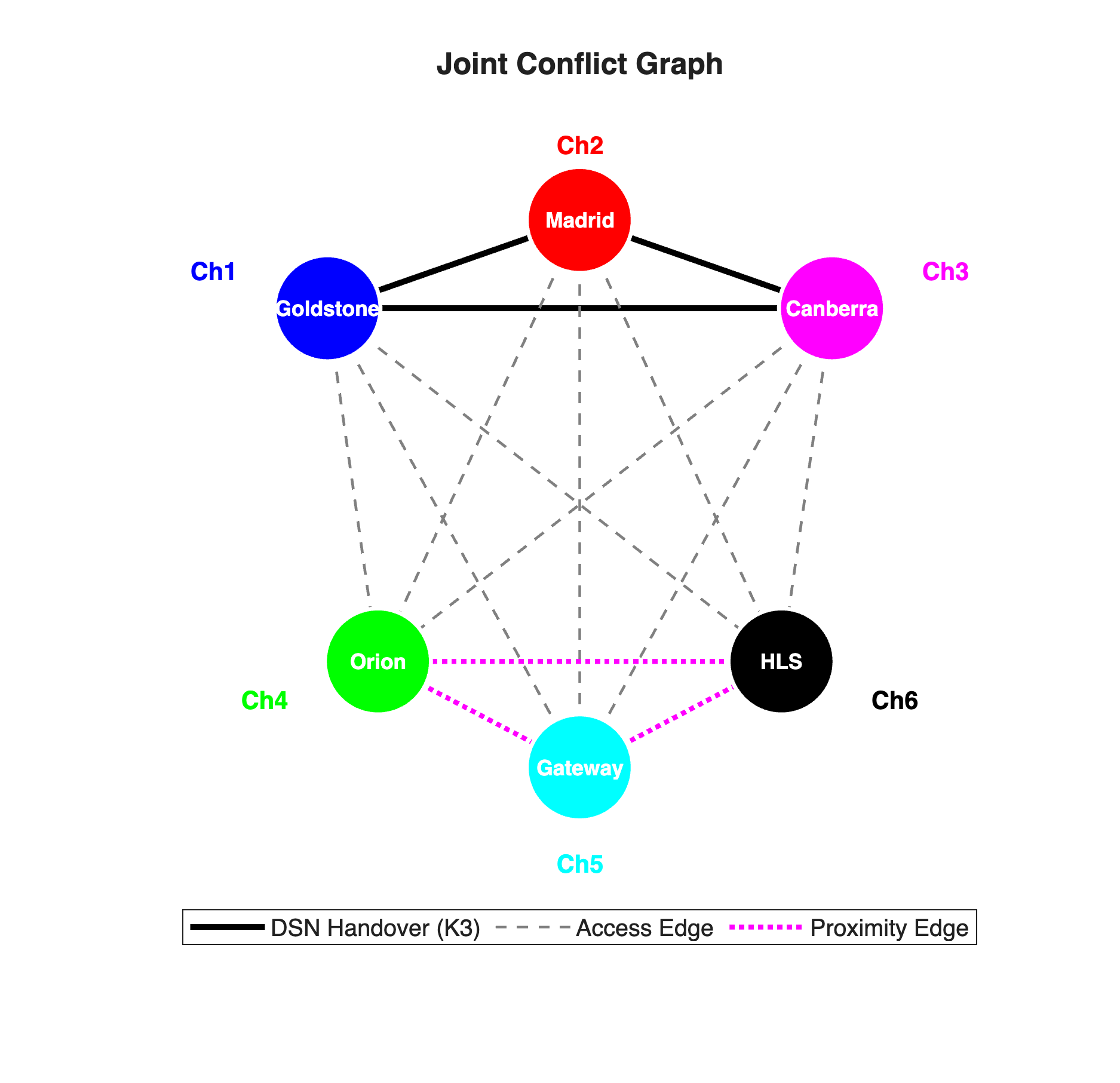}
\caption{Joint ground and space conflict graph.}
\label{fig:joint_intgraph}
\end{figure}

For the physical link budgets, the received S band carrier power was computed at a $2.2$~GHz frequency using the Friis transmission equation. The cumulative telemetry capacity (in Gigabits) extracted across the 10 day timeline is plotted in Fig. \ref{fig:capacity_comparison} for three scheduling configurations.
\begin{itemize}
    \item Stability optimized proactive scheduling with zero retuning, yielding $\simThroughputProactive$~Gb.
    \item Spectrum optimized reactive scheduling, yielding $\simThroughputReactive$~Gb.
    \item Welsh Powell static recomputation, yielding $\simThroughputStatic$~Gb.
\end{itemize}
The stability optimized scheme yields a telemetry throughput gain of $\simThroughputGain$~Gb over the reactive baseline.

\begin{figure}[htbp]
\centering
\includegraphics[width=0.65\linewidth]{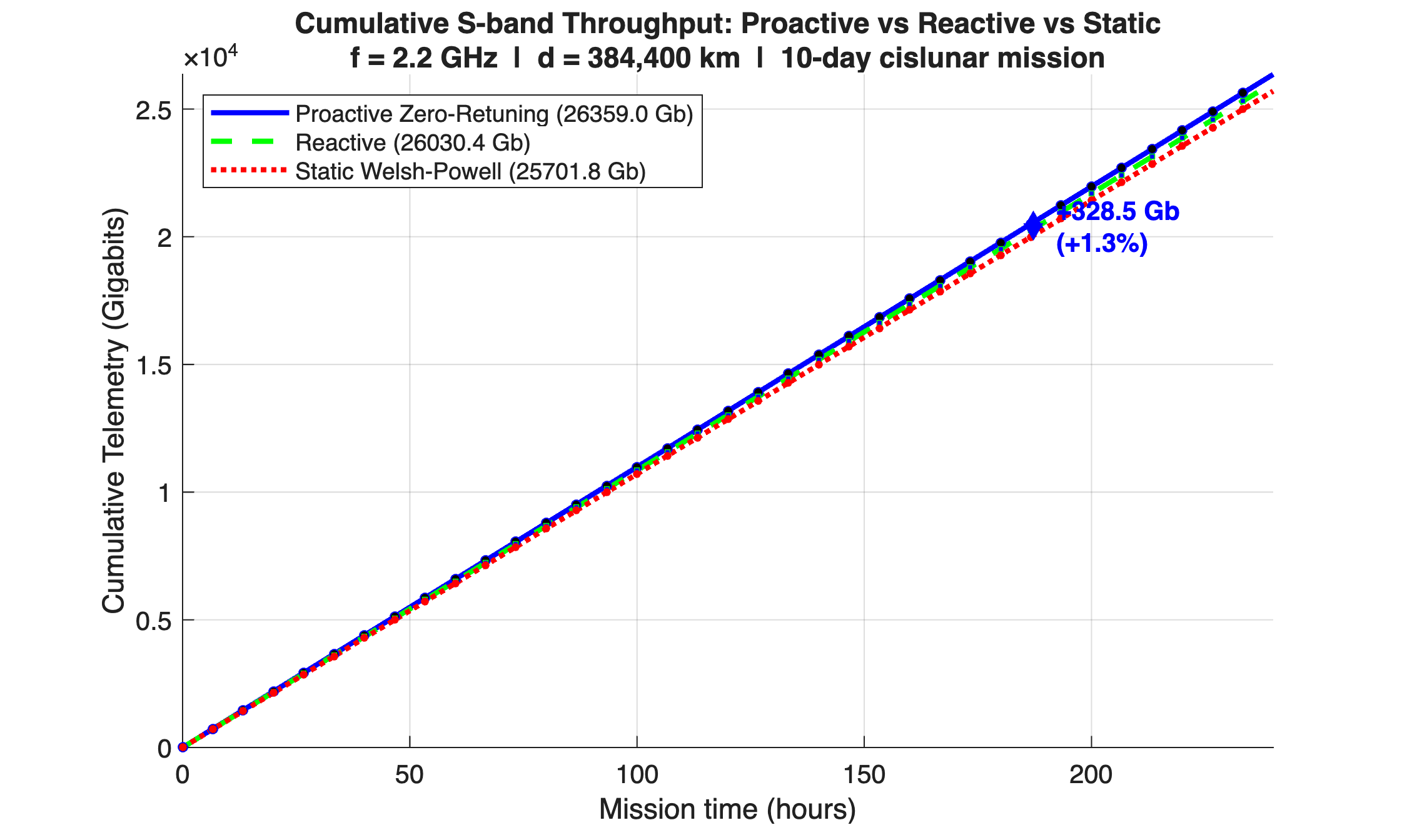}
\caption{Telemetry throughput comparison.}
\label{fig:capacity_comparison}
\end{figure}

Finally, to test terrain masking, we simulated the HLS lander located at the lunar South Pole near Shackleton Crater ($89.9^\circ$S, $0.0^\circ$E). Local topography blocks the direct to Earth link when elevation angles fall below the local horizon mask. The lander experiences exactly \simHlsMaskedSteps{} micro blackout intervals, masking the DTE link for \simHlsMaskedPercent\% of the timeline. The elevation tracking and masking intervals are plotted in Fig. \ref{fig:terrain_masking}.

\begin{figure}[htbp]
\centering
\includegraphics[width=0.72\linewidth]{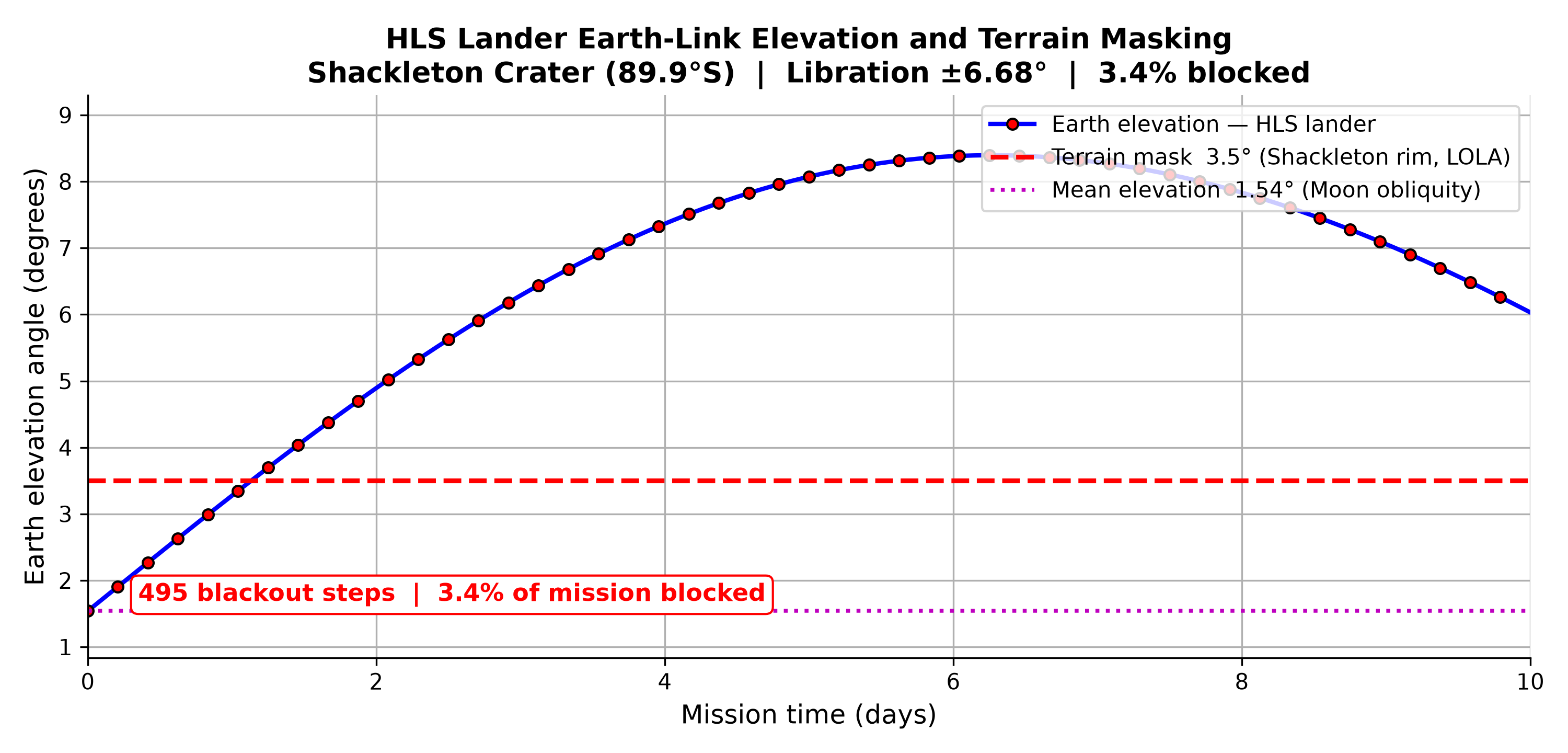}
\caption{Terrain masking at Shackleton Crater.}
\label{fig:terrain_masking}
\end{figure}

\section{Discussion, Literature Benchmarks, and Limitations}
\label{sec:discussion}
The numerical evaluation confirms that offline channel preallocation improves both cumulative data throughput and transponder operational stability relative to reactive baselines.

\subsection{Operational Impact and the Canberra Ground Station Blackout Case Study}

The simulated Day 5 Canberra blackout highlights the operational gap between reactive and proactive scheduling. Under reactive paradigms, an unexpected ground station failure results in immediate signal loss. The spacecraft must then negotiate a new frequency with Madrid or Goldstone, a coordination loop that incurs a 4 to 8 minute telemetry outage~\cite{dsn810005}. Conversely, under the proactive R-SCG model ($\tilde{H}_{\text{int}}$), Madrid operates on a preallocated frequency (color 2) that is strictly nonconflicting with the spacecraft's dedicated channel (color 4). Upon Canberra's failure, the link shifts to Madrid without requiring transponder reconfiguration, preserving continuous data down link.

Similarly, the cumulative throughput curves in Fig. \ref{fig:capacity_comparison} show that eliminating transponder lock drop penalties directly increases throughput. By avoiding the 6 minute cool down of reactive switching and the 12 minute synchronization delay of static re-computation~\cite{dsn810005}, the proactive zero returning scheme recovers $\simThroughputGain$~Gb of telemetry data (a $\simThroughputGainPct$\% increase over the reactive baseline).

\subsection{Comparison and Bench-marking Against Literature}

To validate our findings, we compare the OETGC framework against previous satellite spectrum and scheduling models, specifically the DSN ground network architecture and scheduling of Cheung et al.\ \cite{cheung2015} and the satellite channel assignment of Del Re et al.\ \cite{delre1997}.

\begin{table}[htbp]
\centering
\caption{Methodological Comparison with Existing Literature.}
\label{tab:lit_comparison}
\tiny
\begin{tabular}{@{}llll@{}}
\toprule
Feature / Metric & Prior DSN scheduling \cite{cheung2015} & Satellite reuse \cite{delre1997} & Proposed OETGC model \\
\midrule
Conflict Model & Heuristic resource sharing & Static geographic footprints & Perfect interval or chordal graphs \\
Complexity & NP hard heuristics & NP complete coloring & Polynomial time exact ($O(H \log H)$) \\
Transponder Stability & Not modeled & Not modeled & Zero retuning guaranteed ($k \ge \chi + 1$) \\
Outage Resilience & Reactive replanning & Recomputation needed & Proactive R-SCG recovery \\
Multi-spacecraft & Bounded heuristics & Static planar reuse & Scalable chordal bounds ($3+c$) \\
\bottomrule
\end{tabular}
\end{table}

As summarized in Table~\ref{tab:lit_comparison}, prior DSN contact optimization models \cite{cheung2015} treat scheduling as an NP hard allocation problem, relying on greedy heuristics that do not model frequency stability or transponder returning costs. Del Re's satellite channel reuse models \cite{delre1997} apply static, geographic graph coloring to LEO and GEO grids. While modern approaches increasingly employ machine learning for dynamic routing in satellite constellations, they generally lack the strict, deterministic mathematical guarantees required for zero returning deep space handovers. Collectively, these models struggle to capture the time varying physical constraints of cislunar trajectories, forcing full recalculations during unexpected outages.

Our model solves this by mapping dynamic link conflicts to perfect interval and chordal graphs, producing exact, polynomial time solutions. Unlike prior methods that suffer carrier lock drops during handovers or failures, our approach guarantees zero returning recovery without modifying spacecraft hardware.

Under nominal conditions, cislunar orbits are highly deterministic and straightforward to model using standard SPICE kernels. However, this geometric predictability breaks down during unexpected flight events, such as emergency orbital corrections or autonomous safe mode transitions. When these anomalies alter the spacecraft's trajectory, any precomputed ground schedules immediately become invalid. We integrated the Random Forest classifier specifically to provide onboard autonomy during these critical gaps. Rather than waiting for ground control to calculate and upload new orbital data, the classifier uses real time local sensor data to predict impending handovers instantly. This ensures the spacecraft can maintain its communication links independently until contact with mission control is fully restored.
\subsection{Model Assumptions and Physical Limitations}
Implementing our scheduling framework involves a few physical assumptions:
\begin{enumerate}
    \item \textbf{Stochastic Rain Fades.} We model weather drops as independent Bernoulli trials. In reality, atmospheric attenuation (especially at Ka band) is time correlated, exhibiting clustering behavior (long periods of heavy rain followed by clear sky). Future models should utilize time dependent Markov chain models to capture rain fade duration.
    \item \textbf{Doppler Shift and Guard Bands.} Our model treats the frequency spectrum as discrete mathematical integers. In physical space transceivers, high relative spacecraft velocity introduces significant Doppler shifts, which dynamically alter the occupied bandwidth. To handle these rapid frequency offsets without breaking the precomputed interval boundaries, real world receivers typically rely on hardware level precompensation or blind estimation tracking loops~\cite{ali1998doppler}.
    \item \textbf{Idealized Lunar Geometry.} The orbital simulation assumes a simplified circular orbit with a constant sidereal Earth rotation and $0^\circ$ lunar declination. While real world cislunar trajectories fluctuate by $\pm 28^\circ$ declination (altering the exact start and end times of visibility windows), the underlying $K_3$ and $K_3 + I_m$ topologies remain structurally invariant.
\end{enumerate}

\section{Conclusion}
\label{sec:conclusion}

We presented an analytical graph theoretic framework for resilient spectrum scheduling in cislunar space networks. By mapping temporal tracking sequences to perfect interval graphs and chordal topologies, we demonstrated that a single reserve channel ($k=4$) guarantees zero transponder retuning across nominal handovers and single station outages. For multispacecraft clusters, we established that joint conflict graphs remain chordal, requiring at most $3+c$ channels for $c$ concurrent assets. Orbital simulations verify that eliminating carrier acquisition delays yields a 1.3\% increase in cumulative telemetry throughput. Future research will focus on extending the spectral bounds to time correlated Markovian rain fade models and highly elliptical cislunar orbits.

\section*{Declarations}
\noindent{\textbf{Funding:} The author M Kamrujjaman research was partially supported by the Bose Centre of Advanced Study and Research in Natural Sciences, University of Dhaka,  Bangladesh. 
	
	\medskip
\noindent\textbf{Data Statement:}	No new data were generated in this study. 
	

	\medskip
	\noindent\textbf{Authors' contributions (CRediT):}\\
	\noindent\textbf{Most Esrat Jahan:} Conceptualization, Literature review, Formal analysis, Data curation,  Software, Methodology, Validation,  Writing--original draft.\\
	\textbf{Md. Kamrujjaman:} Conceptualization,  Methodology, Software, Resources,  Investigation, Supervision,  Writing--review \& editing. All authors have read and agreed to the published version of the manuscript.
	
		\medskip
	\noindent\textbf{Competing interests:} The authors declare no competing interests.

	\medskip
	\noindent\textbf{Consent for publication:} No consent is  required to publish this article. 
	
		\medskip
	\noindent\textbf{Generative AI Statement:} The author(s) acknowledge the use of generative AI and text-editing tools,  in the preparation and linguistic refinement of this paper.  After using these tools, the authors 
	reviewed and edited all content as needed and take full responsibility for the content of the 
	published article. 

\bibliographystyle{vancouver}
\bibliography{references}

\end{document}